\documentclass[a4paper,12pt]{article}  
\usepackage{verbatim}  

\usepackage[english]{babel}

\usepackage{amsmath}  
\usepackage{amsthm}  
\usepackage{amssymb}  
\usepackage{bbm}  
\usepackage{bbding}  

\newcommand{\CC}{\mathbbm{C}}  
\newcommand{\EE}{\mathbbm{E}}

\newcommand{\PP}{\mathbbm{P}}

\newcommand{\RR}{\mathbbm{R}}

\newcommand{\1}{\mathbbm{1}}

\newcommand{\dd}{{\rm d}\hspace{-.083em}}  

\swapnumbers  

\theoremstyle{definition}

\newtheorem{Def}{Definition}[section]  

\theoremstyle{plain}

\newtheorem{Avs}[Def]{Claim}  
\newtheorem{Lem}[Def]{Lemma}
\newtheorem{Pro}[Def]{Proposition}
\newtheorem{Thm}[Def]{Theorem}
\newtheorem{Cor}[Def]{Corollary}

\theoremstyle{remark}

\newtheorem{Rmq}[Def]{Remark}
\newtheorem{Xpl}[Def]{Example}

\numberwithin{equation}{section}  

\newenvironment{Idx}{\HandCuffRight\quad\it}{\rm}  

\def\proofname{Démonstration}  

\renewcommand{\[}{\begin{equation}}
\renewcommand{\]}{\end{equation}}
\renewcommand{\epsilon}{\varepsilon}
\renewcommand{\geq}{\geqslant}
\renewcommand{\leq}{\leqslant}
\renewcommand{\limsup}{\varlimsup}

\renewcommand{\to}{\longrightarrow}
\renewcommand{\bar}{\overline}
\renewcommand{\tilde}{\widetilde}
\renewcommand{\hat}{\widehat}
\renewcommand{\vec}{\overrightarrow}

\usepackage{ifthen}
\let\footnotezero=\footnote
\renewcommand\footnote[1]{\footnotezero{~#1}%
\ifthenelse{\arabic{footnote}=3}{\setcounter{footnote}{0}}{}}  

\let\footnotetextzero=\footnotetext
\renewcommand{\footnotetext}[1]{\footnotetextzero{~#1}%
\ifthenelse{\arabic{footnote}=3}{\setcounter{footnote}{0}}{}}

\date{May 7, 2009}
\title{Some ideas about quantitative convergence of collision models to their mean field limit}
\author{R\'emi Peyre}

\newcommand{\mc}{\mathcal}
\newcommand{\Hms}{\smash{\dot{H}}^{-s}}
\newcommand{\da}{\delta\!}
\renewcommand{\|}{|\!|}
\renewcommand{\SS}{\mathbbm{S}}
\newcommand{\egdef}%
{\hspace{2pt}\rlap{\raisebox{.20ex}[0pt][0pt]{$\cdot$}}\raisebox{-.20ex}[0pt][0pt]{$\cdot$}\hspace{-.25em}=}
\newcommand{\mesim}{\raisebox{.5ex}{$_{\#}$}}

\begin{document}

\maketitle

\begin{abstract}
We consider a stochastic $N$-particle model for the spatially homogeneous Boltzmann evolution
and prove its convergence to the associated Boltzmann equation
when $N\to \infty$. For any time $T>0$
we bound the distance between the empirical measure of the particle system
and the measure given by the Boltzmann evolution in some homogeneous negative Sobolev space.
The control we get is Gaussian, i.e.\ we prove that the distance is bigger than
$x N^{-1/2}$ with a probability of type $O(e^{-x^2})$.
The two main ingredients are first a control of fluctuations
due to the discrete nature of collisions, secondly
a Lipschitz continuity for the Boltzmann collision kernel.
The latter condition, in our present setting, is only satisfied
for Maxwellian models.
Numerical computations tend to show that our results are useful in practice.
\end{abstract}

\section*{Introduction}

The Boltzmann equation was written down by L.~Boltzmann
\cite{Boltzmann}\nocite{Boltzmann-traduction} in 1872,
five years after Maxwell's seminal paper~\cite{Maxwell},
to describe the behaviour of a large number of gas molecules
interacting by pairwise collisions. Proving rigorously the heuristic
arguments of Boltzmann to get some convergence
of the $N$-particle model to the continuous Boltzmann equation when $N \to \infty$
is an extremely difficult challenge (it motivated
Hilbert's 6th problem \cite{Hilbert}) that mathematicians are still dealing with.

Here we are only going to handle the \emph{spatially homogeneous} Boltzmann equation
(also called \emph{mean field Boltzmann equation}),
in which one forgets the positions of the gas particles
to concentrate only on the collision phenomenon.
Then proving the convergence of an $N$-particle system
to the continuous equation is a typical \emph{mean field limit} problem:
a particle model is said to be mean field
when each particle interacts with comparable strength with \emph{all} the other ones.
Such a problem, which was first proposed by Kac~\cite{Kac53},
is far more tractable than the original one,
and convergence results, mostly qualitative, have already been obtained for it
(see~\S\,\ref{parcomparaison}).
Here however we are interested in a \emph{quantitative}
version of these results.

The goal of this paper is not only to get some quantitative results of convergence
for spatially homogeneous particle collision models, but
also to get an $N^{-1/2}$ convergence speed,
typical of the uniform central limit theory (see~\cite{Dudley} about that theory),
and to prove some \emph{non-asymptotic} bounds.
Concerning ``real'' Boltzmann models, in the actual state of my work
I am only able to use the results for Maxwellian systems,
and moreover constants in convergence bounds deteriorate rapidly with time.
However that does not seem to be a fundamental feature of my approach,
and I am currently working on further improvements to overcome these issues.

Here is some notation which will be used throughout this paper:
\begin{itemize}
\item The space $\RR^d$ is equipped with its Euclidean structure,
whose norm is denoted by $|\cdot|$.
\item $f : E\to F$ being a measurable function and $\mu$ a measure on $E$,
the image measure of $\mu$ by $f$ on $F$ will be denoted $f\mesim \mu$.
\item $\delta_x$ denotes a Dirac mass at $x$.
\item $\mc{S}(\RR^d)$ is the Schwartz space on $\RR^d$, i.e.\
the set of (complex-valued) $\mc{C}^{\infty}$ functions on $\RR^d$
which tend to $0$ at infinity faster than any $|x|^{-k}$,
as well as all their derivatives.
\item The Fourier transform of a function $f\in\mc{S}(\RR^d)$ is denoted by $\hat{f}$,
with the unitary convention $\hat{f}(\xi) = (2\pi)^{-d/2} \int_{\RR^d} f(x)e^{-i\xi\cdot x}\dd{x}$.
\item The notation $\|\cdot\|$ will be used
to denote Hilbert norms in functional spaces.
If $Q$ is a linear operator between two Hilbert spaces,
its operator norm $\sup_{\|x\|\leq1} \|Qx\|$
will be denoted $|\!|\!|Q|\!|\!|$.
\item $x$, $y$ and $z$ being three points of an affine Hilbert space
with $y,z\neq x$, $\hat{yxz}$ denotes the angle between
$\vec{xy}$ and $\vec{xz}$, which is an element of $[0,\pi]$.
\item The identity matrix of size $d$ is denoted $\mathbf{I}_d$.
\end{itemize}

\section{The model}

\subsection{The microscopic model}

Let us describe the particle model for the spatially homogeneous Boltzmann evolution.
Such models have been first proposed by Kac~\cite{Kac53}
and later thoroughly studied by Sznitman~\nocite{Sznitman-article}\cite{Sznitman-SF}, Spohn~\cite{Spohn} and others.
There are $N$ identical particles indexed by $0,\ldots,N-1$, each particle $i$ being characterized
by its velocity $v_i \in \RR^d$.
One imposes random collision times,
so that the microscopic evolution is a Markov process.
The way two particles with respective velocities $v$ and $w$ hit each other
is described by some positive measure $\gamma_{v,w}$ on $(\RR^d)^2$,
$N^{-1} \dd\gamma_{v,w}(v',w')$ being the collision rate from state $(v,w)$
to state $(v',w')$. In other words,
the generator $\mathcal{L}$ of the Markov process is
\begin{multline}\label{generateurMarkov}
\mathcal{L}f(v_0,\ldots,v_{N-1}) = \\ \frac{1}{2N} \sum_{0 \leq i,j < N} \int_{(\RR^d)^2}
\Big( -f(v_0,\ldots,v_{N-1}) + f(\ldots,v'_i,\ldots,v'_j,\ldots) \Big) \dd{\gamma}_{v_i,v_j}(v'_i,v'_j) .
\end{multline}

We may add to this model some extra physical conditions.
First, we will always suppose that the momentum and energy are conserved by collisions,
and that the model is invariant by velocity translation or rotation,
i.e.\ for all $v,w\in\RR^d$, for any (positive) isometry $J$ of $\RR^d$;
\begin{eqnarray}
\label{conserv-v-micro} \gamma_{v,w}\textrm{-a.e.} && \quad v' + w' = v + w ,\\
\label{conserv-e-micro} \gamma_{v,w}\textrm{-a.e.} && \quad |w'-v'| = |w-v| ,\\
\label{invar-isom-micro} && \gamma_{Jv,Jw} = (J,J)\mesim \gamma_{v,w} .
\end{eqnarray}
When conditions~(\ref{conserv-v-micro}) to~(\ref{invar-isom-micro}) are satisfied,
the model is completely described
by the family of measures $\big(\bar{\gamma}_{u}\big)_{u\in(0,\infty)}$
on $(0,\pi]$, where $\dd\bar{\gamma}_{u}(\theta)$
is the proportion, by unit of time, of particles with relative speed $u$
which undergo a collision making them deviate by an angle $\theta$
in the collision referential.

Moreover, it is often assumed that the $\bar{\gamma}_{u}$ have a scale invariance property,
in the sense that there exists a real parameter $g$ such that for any $\lambda \in (0,+\infty)$,
\[\label{parametre-gamma} \bar{\gamma}_{\lambda u} = \lambda^{g} \bar{\gamma}_{u} .\]
The hard sphere model corresponds to the value $g=1$.
Another very interesting particular case is when $g=0$---then one says
that the model is \emph{Maxwellian}. In this article the concrete results obtained
actually will concern Maxwellian models.

Before turning to the macroscopic model, let us make some remarks on our microscopic model:
\begin{Rmq}\label{RmqsmodeleMarkov}
\begin{enumerate}
\item The $N^{-1}$ factor in Equation~(\ref{generateurMarkov}) is essential to get the mean field limit:
it morally says that the global collision rate of one particle is independent of the total number of particles.
\item Strictly speaking, generator~(\ref{generateurMarkov})
allows a particle to collide with itself, which is physically absurd.
Yet because of the conservation law~(\ref{conserv-e-micro}),
the auto-collision term is actually zero, so there is no problem.
\item\label{itmintblt} The $\gamma_{v,w}$ have to satisfy some integrability conditions for the Markov process
to be well-defined. For instance, if conditions~(\ref{conserv-v-micro}) to~(\ref{parametre-gamma})
are statisfied, then it suffices that for an arbitrarily chosen $u \in (0,\infty)$,
$\int_0^{\pi} \theta^{d-1} \dd\bar{\gamma}_u(\theta)$ is finite~\cite{Sznitman-article}.
\end{enumerate}
\end{Rmq}

\subsection{The macroscopic model}

The macroscopic space-homogeneous Boltzmann equation%
~\cite{Carleman} is obtained informally
by letting $N$ tend to infinity in the microscopic evolution.
Then the particles' velocities are described
by they empirical measure, which is a (possibly non-atomic) probability measure $\mu_t$ on $\RR^d$.
The evolution of that measure is deterministic and is governed by the equation:
\[\label{eq-Boltz} D_t\mu = Q(\mu_t,\mu_t) ,\]
where $Q$ is the \emph{Boltzmann collision kernel} of the system,
formally defined by:
\[ Q(\mu,\nu) =
\frac{1}{2} \int \!\! \bigg( \int\big( -\delta_{v}-\delta_{w}+\delta_{v'}+\delta_{w'} \big) \dd\gamma_{v,w}(v',w') \bigg)
\,\,\dd\mu(v)\dd\nu(w) .\]

Equation~(\ref{eq-Boltz}) is an ordinary differential equation in an infinite-di\-men\-sio\-nal space;
that equation is non-linear because of the quadratic term $Q(\mu,\mu)$.
Unique existence of a solution to it has been thoroughly studied
over the last decades~\cite{Desvillettes, Villani}.
For our theory to work, we will need to work in a setting
where that unique existence is achieved in some convenient space%
---which is quite logical altogether.
Later we will see concrete examples where~(\ref{eq-Boltz})
behaves well for our purpose.

\subsection{Conservation laws, convergence to equilibrium}\label{parCLCTE}

Because of the conservation laws~(\ref{conserv-v-micro})
and~(\ref{conserv-e-micro}), we get $d+1$ invariant functions for the microscopic system:
the first $d$ are synthetised in the momentum $P = \sum_{i=0}^{N-1} v_i$,
and the last one is the energy $K = \frac{1}{2}\sum_{i=0}^{N-1} |v_i|^2$.
In the macroscopic model, these invariants become
$p = \int v \,\dd\mu(v)$ and
$k = \frac{1}{2}\int |v|^2 \,\dd\mu(v)$.
Moreover the fact that the macroscopic model derives
from the description of an evolution of particles
implies two extra properties for it:
first \emph{positivity} of Equation~(\ref{eq-Boltz}),
which means that if $\mu_0$ is a positive measure,
then so are the $\mu_t$ for $t$ positive;
secondly \emph{conservation of mass}
which gives the $(d+2)$-nd invariant $m = \int \dd\mu(v)$
for the macroscopic equation.

Concerning equilibrium, if we impose some minimal non-degeneracy condition
(see~\cite{Villani}),
then it is a well-known beautiful result due to Boltzmann~\cite{Boltzmann}\nocite{Boltzmann-traduction}
that Equation~(\ref{eq-Boltz}) is dissipative for positive measures
and then converges to an equilibrium measure $\mu_{\textrm{eq}}$ depending only on $p$, $k$ and $m$:
for $m=1$ and $p = 0$, it is
\[\label{mueq} \dd\mu_{\textrm{eq}}(v) =  \left( \frac{d}{4\pi k} \right)^{d/2} e^{ -d|v|^2/4k } \dd{v} ,\]
and it has the invariance properties
${\mu_{\textrm{eq}}(p,k,1)} = {\tau_{p}\mesim }{\mu_{\textrm{eq}}(0,k-p^2/2,1)}$,
$\tau_p$ being the translation by vector $p$, and
${\mu_{\textrm{eq}}(\lambda p,\lambda k,\lambda m)} = {\lambda \mu_{\textrm{eq}}(p,k,m)}$.
More recently a beautiful quantitative version of that convergence result
has been proved by Carlen, Gabetta and Toscani~\cite{CGT}.

For the microscopic model, there is also a unique ergodic equilibrium measure
for each value of $P$ and $K$ ($N$ being fixed), which is merely the uniform measure
on the $(dN-d-1)$-dimensional sphere%
\footnote{Possibly of radius $0$.}
of $(\RR^d)^N$ made of $N$-uples of vectors
having these $P$ and $K$. Note that if there are $N$ particles
with momentum $Np$ and energy $Nk$, the marginals of that measure
tend to the continuous equilibrium measure ${\mu_{\textrm{eq}}(p,k,1)}$ when $N\to\infty$.

It is worth recalling that the microscopic process is reversible under its equilibrium measure,
while on the contrary the macroscopic equation~(\ref{eq-Boltz}) exhibits a dissipative behaviour%
---a phenomenon which caused much trouble at Boltzmann's time,
but has been well understood today.

\section{Homogeneous Sobolev spaces}

\subsection{Why homogeneous Sobolev spaces?}\label{Why}

To be able to speak of quantitative convergence we will work in some Banach space.
Which one will we take ? As we want to compare the empirical measure of our particle system
to its limit evolution, a natural choice is to take some coupling distance between measures%
---say, the $W_1$ Wasserstein distance~\cite[\S\,7]{Villani-topics}, defined for $\mu$, $\nu$ two
positive measures with the same mass by:
\[ W_1(\mu,\nu) = \sup_{f\ 1\textrm{-Lip.}} \Big| \int f \dd(\nu-\mu) \Big| ,\]
where ``$f\ 1\textrm{-Lip.}$'' means that the supremum is taken over
all $1$-Lipschitz functions on $\RR^d$.
However it turns out that it is hopeless to get an $N^{-1/2}$ rate of convergence in such a space,
because testing $\nu-\mu$ against so much test functions
makes the uniform central limit theory fail%
---see~\cite[\S\,6.4]{Dudley} for more details.

Thus the idea is to test $\nu-\mu$ against a smaller space made of more regular functions.
Sobolev spaces $W^{s,p}$, $s>0$, are such natural test spaces;
then $\nu-\mu$ will be seen as an element of the dual space $W^{-s,p/(p-1)}$.
For our theory we will have to work in a Hilbert space, so we will take $p=2$
and work in $W^{-s,2} = H^{-s}$;
then we can take $s$ fractional, and it will turn out to be useful indeed.
Yet since defining a norm for $H^{-s}$ spaces requires to
choose some aribtrary length, which is physically annoying,
we will rather consider \emph{homogeneous} $\dot{H}^{-s}$ spaces,
which among other advantages do have a canonical norm.
Let us define these spaces properly.

\subsection{Definition and useful properties}

\begin{Def}
Let $s\in\RR$, and for $f \in \mc{S}(\RR^d)$, set
\[ \|f\|_{\Hms} = \left(\int_{\RR^d} |\hat{f}(\xi)|^2 \, |\xi|^{-2s} \, \dd{\xi} \right)^{1/2} .\]
Then those of the $f \in \mc{S}(\RR^d)$ for which $\|f\|_{\Hms} < \infty$,
equipped with the norm $\|\cdot\|_{\Hms}$, constitute a pre-Hilbert space
with scalar product
\[ \langle f,g \rangle_{\Hms} = \int_{\RR^d} \hat{f}(\xi)\,\bar{\hat{g}}(\xi)\,|\xi|^{-2s}\,\dd{\xi}. \]
The Hilbert space obtained by completing it is denoted $\Hms$.
\end{Def}
\begin{Rmq}
For a physicist, $f:\RR^d\to\CC$ has some homogeneity:
say, the elements in $\RR^d$ are measured in $\mathrm{x}$
(generally $\mathrm{x}$ is a unit of length, say meters) and the elements in $\CC$
are measured in $\mathrm{y}$ (which will often be a density, say $\mathrm{kg}\cdot\mathrm{m}^{-d}$).
Then $\|f\|_{\Hms}$ is measured $\mathrm{y}\cdot{\mathrm{x}}^{s+d/2}$
(in our example, $\|f\|_{\Hms}$ would be measured in $\mathrm{kg}\cdot\mathrm{m}^{s-d/2}$).
Equivalently, if $\mu$ is a measure on $\RR^d$,
the physical dimension of $\|\mu\|_{\Hms}$ is $\mathrm{z}\cdot{\mathrm{x}}^{s-d/2}$,
$\mathrm{x}$ being the physical dimension of the elements of $\RR^d$
and $\mathrm{z}$ the physical dimension of $\mu$ (which in our example would be $\mathrm{kg}$).
\end{Rmq}
As we recalled in~\S\,\ref{Why}, bounding a function or a measure in $\Hms$
means bounding uniformly its integral against some class of regular functions:
\begin{Pro}\label{Hms-comme-dual}
Define $\smash{\dot{H}}^{s}$ in the same way as $\Hms$,
then, for any $f$ for which it makes sense:
\[ \|f\|_{\Hms} =
\sup_{\|g\|_{\smash{\dot{H}}^{s}}\leq 1} \Big| \int_{\RR^d} f(x)\bar{g}(x) \dd{x} \Big| .\]
\end{Pro}

\begin{Lem}\label{Hms-convol}
For $s\in]0,d[$, let $\phi_s$ be the locally integrable function
\[\label{fordefus} \phi_s(x) = |x|^{-(d-s)} ,\]
then one has for all $f,g\in\mc{S}(\RR^d)$:
\[\label{forpfffatigue}
\langle f,g \rangle = c(s,d)^2 \big\langle f\ast \phi_s , g\ast \phi_s \big\rangle_{L^2(\RR^d)} ,\]
with
\[ c(s,d) = \frac{\Gamma\big( (d-s)/2 \big)}{(2\pi)^{d/2} \Gamma (s/2)} ,\]
$\Gamma(\cdot)$ being Euler's Gamma function.
\end{Lem}
\begin{proof}
Use that the Fourier transform of $|\xi|^{-s}$
is $(2\pi)^{d/2}c(s,d)\phi_s(x)$, cf.~\cite[exercise \hbox{V-10}]{Schwartz}.
\end{proof}

\begin{Avs}
Let $J_{\lambda}$ be a similarity of $\RR^d$ with dilation factor $\lambda$,
then for any map $f \in \Hms$,
\[ \| f\circ J_{\lambda} \|_{\Hms} = \lambda^{s+d/2} \| f \|_{\Hms} .\]
Equivalenty, for any measure $\mu\in\Hms$,
\[ \| J_{\lambda}\mesim \mu \|_{\Hms} = \lambda^{s-d/2} \| \mu \|_{\Hms} .\]
\end{Avs}

\begin{Idx}
From now on, we will always implicitly write $s = d/2+r$.
\end{Idx}

\begin{Pro}
Suppose $d\geq 2$\footnote{The proposition remains valid with $d=1$,
except that it must be demanded that $r < 1/2$.}
and let $\mu$ be a compactly supported signed measure on $\RR^d$ with total mass $0$,
then for any $r \in (0,1)$, $\mu$ can be seen as an element of $\Hms$.
\end{Pro}
\begin{proof}
Thanks to Lemma~\ref{Hms-convol} we just need to prove
that $\mu\ast \phi_s$ is a square-integrable function.
Suppose that $\mu$ is supported by the ball $B(R)$
of radius $R$ centered at $0$
and splits into $\mu_+-\mu_-$ with $\mu_+$ and $\mu_-$
positive measures each of total mass $M$.
Then for $\rho>0$, on $B(\rho)$, $\mu\ast \phi_s$
is equal to $\mu\ast (\1_{B(R+\rho)} \phi_s)$,
so the $L^2$ norm of $\1_{B(\rho)} (\mu\ast \phi_s)$ is bounded above
by $2M \cdot \|\1_{B(R+\rho)} \phi_s\|_{L^2} < \infty$.
Thus $\mu\ast \phi_s$ is locally $L^2$.
On the other hand, for $|x| = \rho > R$,
\[ |(\mu\ast \phi_s)(x)| \leq M \left( \frac{1}{(\rho-R)^{d/2-r}} - \frac{1}{(\rho+R)^{d/2-r}} \right)
\leq 2MR \, \frac{d/2-r}{(r-R)^{d/2+1-r}} ,\]
so $\mu\ast \phi_s$ is $L^2$ at infinity, which finishes the proof.
\end{proof}

\begin{Cor}
Still suppose $d\geq 2$, then for $r \in (0,1)$,
any signed measure with zero total mass, if it has an $r$-th polynomial momentum,
can be seen as en element of $\Hms$.
\end{Cor}
\begin{proof}
Let $\mu = \mu_+-\mu_-$ be such a measure with its Hahn decomposition,
$\mu_+$ and $\mu_-$ each having total mass $M$.
Then the integral Minkowski inequality gives
\begin{multline}
\|\mu\|_{\Hms} \leq \frac{1}{M} \int_{(\RR^d)^2} \| \delta_x - \delta_y \|_{\Hms} \dd\mu_+(x)\dd\mu_-(y) \\
= \frac{C_r}{M} \cdot \int_{(\RR^d)^2} |x-y|^r \dd\mu_+(x)\dd\mu_-(y) < \infty ,
\end{multline}
$C_r$ being the $\Hms$ norm of any $\delta_x - \delta_y$ for $|x-y|=1$,
which is finite by the previous proposition.
\end{proof}

\begin{Rmq}
Note that the $\Hms$ norm allows us to measure the \emph{distance}
between two (sufficiently integrable) probability measures,
however speaking of the $\Hms$ norm of a \emph{single} probability measure would be nonsense!
Note also that, by Sobolev imbedding,
one can bound above $\|\nu-\mu\|_{\Hms}$,
for any two probability measures $\mu$ and $\nu$,
by (up to some explicit multiplicative constant)
\[ W_{1,r}(\mu,\nu) = \sup \big\{ \big| \int f\dd{\mu} - \int f\dd{\nu} \big| \ ;\ \forall x,y \ |f(x)-f(y)| \leq |y-x|^r \big\} .\]
\end{Rmq}

\section{Dynamic control}

\subsection{Abstract setting}

Now let us study the evolution of our particle system along time.
For the sake of elegance we are going to state our results
in an abstract setting first.

Let $H$ be a Hilbert space (to be thought of as $\Hms(\RR^d)$),
and let $(\hat{X}_t)_{t\geq0}$ be some jump Markov process on that Hilbert space
($\hat{X}_t$ has to be thought of as $\hat{\mu}^N_t$, the empirical measure of process%
~(\ref{generateurMarkov})), with generator $\mc{L}$.
\emph{Stricto sensu} $\mc{L}$ acts on some space of \emph{real} functions on $H$,
say the space of continuous bounded functions $\mc{C}_b(H,\RR)$,
but its definition can be straightforward generealized
to the space $\mc{C}_b(H,E)$ for any Banach space $E$, defining the operator $\mc{L}^{(E)} :
\mc{C}_b(H,E) \to \mc{C}_b(H,E)$ through:
\[ \forall \phi \in E' \ \ \forall f \in \mc{C}_b(H,E) \qquad
\big\langle \phi, \mc{L}^{(E)}f \big\rangle
= \mc{L} \big(\langle\phi,f\rangle\big) .\]
By abuse of notation we will still denote $\mc{L}$ for $\mc{L}^{(E)}$,
thus giving a meaning to expressions like $\mc{L}I$, $I$ being the identity on $H$%
\footnote{In that case $E=H$.}.
With that notation, define $(X_t)_{t\geq0}$ as the deterministic process on $H$
following the differential equation:
\[\label{moyenneMarkov} D_t X = \big(\mc{L}I\big) (X_t) .\]
Equation~(\ref{moyenneMarkov})
has to be thought of as~(\ref{eq-Boltz}).

Our goal is to control the distance between $\hat{X}_t$ and $X_t$.
Here what is important for us is to have a good control
of large deviations for that distance.
As Cram\'er's method cannot be applied directly
because of the infinite-dimensional setting,
we introduce an exponential utility function $\mc{U} : H\to\RR$
defined as
\[ \mc{U}(x) = e^{\|x\|} + e^{-\|x\|} .\]
The following claim gathers the properties of $\mc{U}$
we will use in our work:
\begin{Avs}\label{proprietes-de-U}
\begin{enumerate}
\item\label{minoU} For all $x\in H$, $\mc{U}(x) \geq e^{\|x\|}$;
\item $\mc{U}(0)=2$;
\item\label{UunpeuLip} For all $x,h\in H$, $\mc{U}(x+h) \leq e^{\|h\|}\mc{U}(x)$;
\item $\mc{U}$ is of class $\mc{C}^{\infty}$%
\footnote{To see it, note that $\mc{U}(x) = f(\|x\|^2)$
where $f = 2\cosh(\sqrt{\cdot})$ is $\mc{C}^{\infty}$
on the closed interval $[0,+\infty[$---actually $f$ can be extended analytically to the whole of $\RR$.};
\item\label{gradientU} For all $x\in H$, $\nabla\mc{U}(x)$ is positively colinear to $x$;
\item\label{D2U} For all $x\in H$, $|\!|\!| \nabla^2\mc{U} |\!|\!| \leq \mc{U}(x)$.
\end{enumerate}
\end{Avs}

Then one can state the theorem which will be our central tool.
We first need some notation to alleviate our formulas:
\begin{Def}
We denote $e_1(t) = (e^t-1)/t$, extended by $e_1(0)=1$,
resp. $e_2(t) = (e^t-1-t)/t^2$, extended by $e_2(0)=1/2$.
We also denote $\kappa_-$ the negative part of $\kappa$,
i.e.\ $\kappa_-=\max\{-\kappa,0\}$.
\end{Def}

\begin{Thm}\label{ZEthm}
Suppose that Equation~(\ref{moyenneMarkov}) has a $\kappa$-contracting semigroup
for some $\kappa\in\RR$, in the sense that for all $x,h \in H$:
\[\label{hyp-reg} \big\langle D_x(\mc{L}I) \cdot h , h \big\rangle \leq -\kappa \|h\|^2 . \]
Suppose moreover that the Markov process---which we recall to be a jump process---%
has the amplitude of all its jumps bounded above by some $L<\infty$,
and satisfies:
\[\label{defdeV} \forall x \in H \quad \mc{L}(\| \cdot - x \|^2)(x) \leq V \]
for some $V < \infty$.

Then, denoting $\hat{X}_0$ the (random) initial value of the Markov process
and $X_0$ the (deterministic) initial value of the differential equation~(\ref{moyenneMarkov}),
one has for any $T \geq 0$, for any $\lambda > 0$:
\begin{multline}\label{forDUthm}
\ln \EE\big[\mc{U}\big(\lambda(\hat{X}_T - X_T)\big)\big] \\
\leq \ln \EE\big[\mc{U}\big(\lambda e^{-\kappa T}(\hat{X}_0 - X_0)\big)\big]
+ \lambda^2 e_2(\lambda e^{2\kappa_-T}L) e_1(-2\kappa T) VT.
\end{multline}
\end{Thm}

\begin{proof}
The principle of the proof is merely to show that some time-depending functional
\[ F(\hat{X}_t) = e^{h(t)} \mc{U} \Big( \lambda e^{\kappa(t-T)} (\hat{X}_t - X_t) \Big) ,\]
for a well-chosen function $h$, is a supermartingale.

\begin{Idx}
To make our computations completely rigorous, throughout the proof
we will assume that the expected number of collisions per unit of time
is uniformly bounded, that is, that there is some $M < \infty$ such that
$|(\mc{L}\1_{A})(x)| \leq M$ for all Borel subset $A\subset H$ and all $x\in H$.
Then the general result can be recovered by a standard truncation argument.
\end{Idx}

Let us fix some $t \in [0,T]$ and suppose $(\hat{X}_{t'})_{t'\in[0,t]}$ is known.
Let $\da{t}$ be a small amount of time devised to tend to $0$;
$O({\da{t}}^n)$ will denote any quantity
bounded by some $C{\da{t}}^n$
when $\da{t}$ tends to $0$, where $C$ depends only on
$\kappa$, $V$, $M$, $\lambda$, $T$, $t$ and $\|X_t\|$.

With this notation, the law of $\hat{X}_{t+\da{t}}$ depends on $(\hat{X}_{t'})_{t'\in[0,t]}$
only through $\hat{X}_t$,
and our goal is to show that $\EE[ F(\hat{X}_{t+\da{t}})] - F(\hat{X}_t)$,
which is $O(\da{t})$, is nonpositive%
---more precisely, we only need to prove that
$\EE[ F(\hat{X}_{t+\da{t}})] - F(\hat{X}_t) \leq O({\da{t}}^2)$%
\footnote{Beware that ``$\textit{expr.} \leq O(\da{t}^n)$''
does not mean ``$\textit{expr.} = O(\da{t}^n)$''
but actually ``$(\textit{expr.})_+ = O(\da{t}^n)$''.}.

Set $\hat{Y} = \hat{X} - X$. Denote $\da{\hat{X}} = \hat{X}_{t+\da{t}} - \hat{X}_t$, resp.\
$\da{X} = X_{t + \da{t}} - X_t$, $\da{\hat{Y}} = \hat{Y}_{t+\da{t}} - \hat{Y}_t$,
$\da{F} = F(\hat{X}_{t+\da{t}}) - F(\hat{X}_t)$.
The fundamental observation is that
\[\label{observation-centrage} \EE\big[\da{\hat{X}}\big] = \big(\mc{L}I\big) (\hat{X}_t) \,\da{t} + O(\da{t}^2) .\]
Now, admitting temporarily that $h$ will be of class $\mc{C}^2$, we write:
\begin{eqnarray}
\da{F} &=& h'(t)F(t)\da{t} \\
\label{terme-k} &+& e^{h(t)} \lambda e^{\kappa(t-T)} \nabla\mc{U}(\lambda e^{\kappa(t-T)} \hat{Y}_t) \cdot
\big( \mc{L}I(\hat{X}_t) - \mc{L}I(X_t) + \kappa \hat{Y}_t \big) \da{t} \\
\label{terme-fluc}
& \hspace{-4em} + & \hspace{-2em} e^{h(t)} \Big[ \mc{U} \big(\lambda e^{\kappa(t-T)} \hat{Y}_{t+\da{t}}\big)
- \mc{U} \Big( \lambda e^{\kappa(t-T)} \{ \hat{Y}_t + [\mc{L}I(\hat{X}_t) - \mc{L}I(X_t)] \da{t} \} \Big) \Big] \\
\nonumber &+& O({\da{t}}^2).
\end{eqnarray}

In that sum we first see that the term~(\ref{terme-k}) is nonpositive:
(\ref{hyp-reg}) indeed implies, for all $x,y\in H$,
\[ \big\langle (\mc{L}I)(x+y) - (\mc{L}I)(x) + \kappa y \ ,\  y \big\rangle \leq 0 ,\]
which we apply here with $x = X_t$ and $y = \hat{Y}_t$,
using that $\nabla\mc{U} (\lambda e^{\kappa(t-T)} \hat{Y}_t)$ is positively colinear to $\hat{Y}_t$
(Claim~\ref{proprietes-de-U}-\ref{gradientU}).

Now let us look at term~(\ref{terme-fluc}).
Because of~(\ref{observation-centrage}), the expectation of the random variable
\[\label{VA-fluc}
\lambda e^{\kappa(t-T)} \Big( \hat{Y}_{t+\da{t}} - \big( \hat{Y}_t + [\mc{L}I(\hat{X}_t) - \mc{L}I(X_t)] \da{t} \big) \Big) \]
is $O(\da{t}^2)$.
We will use it thanks to the following
\begin{Lem}\label{lemme3.3}
Let $X \in H$; let $y$ be a random variable with zero mean.
Then one has:
\[\label{forlemme3.3} \EE[\mc{U}(X+y)] \leq \mc{U}(X) \big( 1+ \EE \big[e_2(\|y\|) \, \|y\|^2 \big] \big) .\]
\end{Lem}
{\def\proofname{Proof of the lemma}
\begin{proof}
Taylor's formula yields
\[\label{Taylor} \mc{U}(X+y) = \mc{U}(X) + \nabla\mc{U}(X)\cdot y +
\left( \int_0^1 (1-\theta) \nabla^2\mc{U}(X+\theta y) \dd{\theta} \right) \cdot (y\otimes y) .\]
In that sum the third term is bounded above by
\[ \|y\|^2 \int_0^1 (1-\theta)\mc{U}(X+\theta y) \dd\theta \]
by Claim~\ref{proprietes-de-U}-\ref{D2U}, which in turn is bounded by
\[ \|y\|^2 \, \mc{U}(X) \int_0^1 (1-\theta) e^{\theta \|y\|} \dd{\theta} \ = \ e_2(\|y\|) \, \|y\|^2 .\]
by Claim~\ref{proprietes-de-U}-\ref{UunpeuLip}.
Taking expectation gives the result since the second term in sum~(\ref{Taylor})
has zero mean by assumption.
\end{proof}}
What does it give for us? Let $A$ be the event
``some collision occurs between $t$ and $t+\da{t}$''.
$A$ is an event of probability $O(\da{t})$,
on $A$, (\ref{VA-fluc}) is $O(1)$,
and on $^{c\!\!}A$ it is $O(\da{t})$.
Hence, denoting temporarily $\ast$ for that random variable,
$\EE\big[ {\|\ast\|^2} e_2(\|\ast\|) \big]$, up to some $O(\da{t}^2)$,
is merely $\lambda^2 e^{2\kappa(t-T)} \EE\big[ \| \da{\hat{Y}} \|^2 e_2(\|\lambda e^{\kappa(t-T)} \da{\hat{Y}}\|) \big]$,
which is bounded above by
$\lambda^2 \linebreak[2] e^{2\kappa(t-T)} \linebreak[2] e_2(\lambda e^{2\kappa_-T} L) V$
uniformly in $t$.

Putting all things together, we get
\[ \EE[\da{F}] \leq \Big( h'(t) + \lambda^2 e^{2\kappa(t-T)} e_2(\lambda e^{2\kappa_-T} L) V \Big) F(t) \da{t} + O(\da{t}^2) ,\]
which will be $\leq O(\da{t}^2)$ provided
\[ h'(t) \leq -\lambda^2 e^{2\kappa(t-T)} e_2(\lambda e^{2\kappa_-T} L) V. \]
To achieve that optimally with $h(T)=0$, we choose
\[ h(t) = \lambda^2 e_2\big(\lambda e^{2\kappa_-T} L\big) e_1\big( 2\kappa(t-T) \big) V (T-t) ,\]
which is of class $\mc{C}^2$ indeed.
Formula~(\ref{forDUthm}) then follows by the supermartingale property.
\end{proof}

\begin{Rmq}
Strictly speaking our proof only shows that $F(\hat{X}_t)$ is a \emph{local} supermartingale.
But this local supermartingale is nonnegative, so it is actually a global supermartingale
(see~\cite[\S\,IV-1.5]{RY}).
\end{Rmq}

\subsection{Application to Boltzmann's model}\label{ApplThmBoltz}

To apply Theorem~\ref{ZEthm} to our Boltzmann model,
we have to compute the values of $L$, $V$, $\kappa$ and
$\EE\big[\mc{U}\big(\lambda e^{-\kappa T}(\hat{X}_0 - X_0)\big)\big]$.
In this subsection let us just look at the first two quantities%
---the last two ones will be the objects of separate sections.

\begin{Idx}
From now on, when dealing with Boltzmann models
we are working in the space $\Hms(\RR^d)$ for some $r\in(0,1)$.
We denote by $C_r$ the $\Hms$ norm of any $\delta_x - \delta_y$ for $|x-y|=1$,
which is some finite explicit function of $d$ and $r$.
\end{Idx}

Recall that $K$ denotes the energy of the $N$-particle system,
which is conserved along the stochastic evolution%
---note by the way that up to translating the origin of $\RR^d$,
we can replace $K$ by the internal energy
\[ \tilde{K} =  K - \frac{|P|^2}{2N} .\]
Then at any time no particle has speed greater than
$\sqrt{2K}$, so the effet of a collision between two particles
on the empirical measure can be no more than
$2\cdot(8K)^{r/2}C_r N^{-1}$,
which gives an admissible value for $L$.
\begin{Rmq}\label{Rmqraffin}
To get the bound $L \leq 2\cdot(8K)^{r/2}C_r N^{-1}$ we have used that
the relative speed between two particles is at most $2\sqrt{2K}$,
and that the effect of a collision with relative speed $u$ is at most $2u^r N^{-1}$.
Actually one can do better:
the relative speed between two particle is at most $2\sqrt{K}$
and the effect of a collision with relative speed $u$ is at most $2\sqrt{2^{1-r}-1} \, C_r u^r/N$,
so we could have taken
\[\label{bonnevaleurL} L = 2^{1+r} \sqrt{2^{1-r}-1} \, C_r K^{r/2} N^{-1} .\]
It is that bound that we will use in the sequel.
As it does not change qualitatively the results compared to
the rough value $L = 2\cdot(8K)^{r/2}C_r N^{-1}$,
I let the proof of~(\ref{bonnevaleurL}) as an exercise for the reader.
\end{Rmq}
Anyway remember that, since $K$ is going to be
of order of magnitude $O(N)$, one has $L=O(N^{r/2-1})$
when $N\to\infty$.

Now let us compute $V$, which has been defined by~(\ref{defdeV}).
$V$ must be a bound for the expectation by unit of time of the square of the jumps done by the process $\hat{X}$,
which, denoting $\hat{X} = \mu$ for the Boltzmann model, is bounded above by
\begin{multline}
2{C_r}^2 N^{-1} \int_{(\RR^d)^2} |w-v|^{2r} \dd\mu(v)\dd\mu(w) \\
\stackrel{\textrm{Jensen}}{\leq} 2{C_r}^2 N^{-1} \left( \int_{(\RR^d)^2} |w-v|^2 \dd{\mu}(v)\dd{\mu}(w) \right)^{r} \\
= 2^{1+2r} {C_r}^2 N^{-1} \left(\frac{\tilde{K}}{N}\right)^r \leq 2^{1+2r} {C_r}^2 K^r N^{-(1+r)}.
\end{multline}
Taking into account Remark~\ref{Rmqraffin}, one can even take
\[\label{bonnevaleurV} V =  (2^{1-r}-1) 2^{1+2r} {C_r}^2 K^r N^{-(1+r)} .\]
Anyway remember that $V=O(N^{-1})$ when $N\to\infty$.

\subsection{Comments on the results}\label{parcommdyna}

\begin{Idx}
All the computations in this subsection are purely heuristic,
so we will drop lower order terms without wondering when we can do so.
$C_1, C_2, \ldots$ will denote constants depending only on
$\kappa$, $V$, $L$ and $T$, whose exact expression does not interest us.
\end{Idx}

In the right-hand side of Formula~(\ref{forDUthm}) there are two terms:
the first one, $\ln \EE[\mc{U}(\lambda e^{-\kappa T}(\hat{X}_0 - X_0))]$,
merely expresses the difference between the experimental initial condition
and its continuous limit. There is obviously no surprise in getting such a term,
whose study is deferred to~\S\,\ref{flucCI}. For the time being just notice
the presence of the factor $e^{-\kappa T}$ in front of $\hat{X}_0 - X_0$, which means
that the effect of initial fluctuations will be quite large if $\kappa<0$,
and conversely quite small if $\kappa>0$.

The actual dynamic effect in~(\ref{forDUthm}) lies in the term
$\lambda^2 \linebreak[2] e_2(\lambda e^{2\kappa_-T}L) \linebreak[2] e_1(-2\kappa T) \linebreak[2] VT$.
Let us study it in the case of our Boltzmann model,
according to~\S\,\ref{ApplThmBoltz}.
We have noticed that, when $N$ becomes large, one has
$L = O(N^{r/2-1})$, resp.\ $V = O(N^{-1})$.
So let us write $L\simeq\ell N^{r/2-1}$,
resp.\ $V\simeq\omega N^{-1}$. Then the dynamic term of~(\ref{forDUthm}) becomes:
\[\label{approx1}
\lambda^2 e_2(\lambda e^{2\kappa_-T}L) e_1(-2\kappa T) VT \simeq
\lambda^2 N^{-1} e_2(\lambda e^{2\kappa_-T}\ell N^{r/2-1}) e_1(-2\kappa T) \omega T.\]
The $\lambda^2 N^{-1}$ factor hints that
the good order of magnitude for $\lambda$ will be $\lambda = O(N^{1/2})$.
So write $\lambda = y N^{1/2}$; then~(\ref{approx1}) becomes
\[\label{approx2} \lambda^2 e_2(\lambda e^{2\kappa_-T}L) e_1(-2\kappa T) VT \simeq
e_2(y e^{2\kappa_-T}\ell N^{(r-1)/2}) e_1(-2\kappa T) \omega y^2 T  .\]
In our case $(r-1)/2 < 0$ so,
if $N$ is sufficiently large, $y e^{2\kappa_-T}\ell N^{(r-1)/2}$
is very close to zero and the $e_2(\ast)$ term
is very close to $e_2(0)=1/2$, finally giving
\[\label{approx3} \lambda^2 e_2(\lambda e^{2\kappa_-T}L) e_1(-2\kappa T) VT \simeq
\frac{1}{2} e_1(-2\kappa T) \omega y^2 T . \]

For a fixed $T$, (\ref{approx3}) shows that the dynamic term
in Formula~(\ref{forDUthm}) is approximately $C_1 y^2$.
Moreover, as we will see in~\S\,\ref{flucCI}, the static term
$\ln \EE[\mc{U}(\lambda e^{-\kappa T}(\hat{X}_0 - X_0))]$
is approximately $C_2 y^2 + C_3$.
In the end, one gets
\[\label{for783} \ln \EE\big[\mc{U}\big(y N^{1/2} (\hat{X}_t - X_t)\big)\big] \lesssim C_4 y^2 + C_3 ,\]
hence by Markov's inequality and Claim~\ref{proprietes-de-U}-\ref{minoU},
for all $x>0$,
\[\label{for688} \PP\Big( yN^{1/2} \|\hat{X}_T-X_T\| \geq x \Big) \lesssim e^{C_4 y^2 + C_3 - x} .\]
Optimizing Formula~(\ref{for688}) for fixed $x/y$ ratio, one finally finds:
\[\label{forctrlgauss}
\forall \epsilon \geq 0 \quad \PP\big( \|\hat{X}_T-X_T\| \geq \epsilon \big) \lesssim \exp\big(C_3-C_5N\epsilon^2\big) .\]
So Theorem~\ref{ZEthm} applied to the Boltzmann model
gives a Gaussian control for the fluctuations between $\hat{X}_T$ and $X_T$
for any fixed value of $T$---provided the existence of some contractivity constant $\kappa$,
which will be proved for the Maxwellian case in~\S\,\ref{parcontractivite}.
Moreover the order of magnitude of the fluctuations we get is $N^{-1/2}$,
the typical deviation size in central limit theorems.
So we may say that the bounds we have got
are a kind of explicit dynamic central limit bound
for the Boltzmann model.

\begin{Rmq}
Actually the approximations we did to get~(\ref{for783}) are sensible only if
$y$ is not too large, otherwise $\lambda e^{2\kappa_-T} \ell N^{r/2-1} \gtrsim 1$
and then the $e_2(\ast)$ term in~(\ref{approx1}) cannot be considered as
close to $1/2$.
It follows that our computations
are valid only for $\lambda \lesssim N^{1-r/2}/\ell$,
i.e.\ for $y \lesssim N^{(1-r)/2}/\ell$.
Tracking that constraint throughout our reasoning,
it finally turns out that~(\ref{forctrlgauss})
is only valid for $\epsilon \lesssim T\omega N^{-r/2}/\ell$.
So our Gaussian control does not hold up to large deviations
but only to intermediate deviations.
Fortunately~(\ref{forctrlgauss}) tells us
that the probability of such intermediate deviations
is bounded above by something like $e^{-C_6N^{1-r}}$, which goes very fast to $0$ anyway.
Moreover, even for $\epsilon \gg \omega T N^{-r/2}/\ell$
one can still use~(\ref{for688}) with $y = N^{(1-r)/2}/\ell$
and $k \simeq yN^{1/2}\epsilon$, which gives an exponential control
of the tail of the law of $\| \hat{X}_T - X_T \|$
applicable to large deviations.
\end{Rmq}

The behaviour of Formula~(\ref{approx3}) as $T$ becomes large
depends on the value of $\kappa$:
\begin{itemize}\label{selonvaleurk}
\item If $\kappa < 0$ (the worst case), then the $e_1(-2\kappa T)$ factor becomes
exponentially large as soon as $T \gtrsim 1/|\kappa|$.
Thus the dynamic control given by Theorem~\ref{ZEthm}
is relevant only for moderate values of $T$ corresponding to durations
for which each particle makes only a couple of collisions.
Moreover, as we noticed in the beginning of that subsection,
in that case the term due to the control of initial fluctuations
will also become huge as $T$ increases.
Note however that qualitatively we get a Gaussian control
for \emph{any} fixed $T$, only the constants in that control becoming bad.
\item If $\kappa = 0$ the dynamic term of~(\ref{forDUthm})
increases proportionally to $T$, so our bound remains good
even for moderately large values of $T$, but ultimately gets uninteresting.
\item\label{cask>0} If $\kappa > 0$ (the best case), then $Te_1(-2\kappa T) \sim 1/2\kappa$ when $T\to\infty$
so the right-hand side of~(\ref{forDUthm})
remains bounded uniformly in $T$,
implying that the $N$-particle model approximates well its continuous limit for \emph{any} time%
\footnote{\label{notek>0}Of course, it does not mean that \emph{one}
random particle system has large probability to stay \emph{always}
close to the continuous limit---which is trivially false by ergodicity---%
but actually that at any \emph{given} time, \emph{most} of the particle systems
will be close to the limit.}. Note that $\kappa>0$ is tantamount to having
an exponential convergence of~(\ref{eq-Boltz}) to equilibrium
in $\Hms$, so in that case our bound rather looks like
a result of convergence to ``equilibrium'' for the empirical measure $\hat{\mu}^N_t$.
\end{itemize}

\section{Contractivity of the collision kernel}\label{parcontractivite}

\subsection{Limitations due to our settings}

In this section we are going to look for computing constant $\kappa$ in~(\ref{hyp-reg}).
Unfortunately it turns out that, for the choices we have made,
our results are unavoidably limited, as we quickly explain in this foreword.
Let me stress however that all the issues encoutered may be solved
by working in a trickier space than the mere $\dot{H}^{-s}$ space.

First $\kappa$ can only be negative, which is the worst case
(see page~\pageref{selonvaleurk}).
Why that? Well, if $\kappa$ were positive, as we said previously
it would imply convergence of Equation~(\ref{eq-Boltz}) to a unique equilibrium
for all probability measures. Yet there are several different
equilibrium probability measures for the Boltzmann evolution
(see Formula~(\ref{mueq}) and below), whose differences lie in $\Hms$,
which is a contradiction. So $\kappa$ is nonnegative.
Then we could prove, using that the model is nondegenerate,
that $\kappa$ cannot be zero and thus is positive.
To have a chance to get negative values of $\kappa$,
$\Hms$ should be replaced by a Banach space which only contains signed measures $\eta$
such that $\int \eta(\dd{x}),\ \linebreak[2]\int x\eta(\dd{x}),\ \linebreak[2]\int |x|^2\eta(\dd{x}) = 0$%
---but which one?

Secondly, the only chance for $\kappa$ to be finite is the case of Maxwellian models
(remember definition below~(\ref{parametre-gamma})):
as will be seen later, this is due to a bad scale invariance property
for non-Maxwellian models.
Though the Maxwellian case is often a useful first step for theoretists,
the physical models encountered in real life do not have any
reason for being so!
To have a chance to get results for non-Maxwellian models,
$\Hms$ should be replaced by some non-homogeneous space%
---but non-homogeneous spaces are often less tractable than homogeneous spaces
and more difficult to interpret physically.

\subsection{Principle to the computation of $\kappa$}

Some readers may have jumped when reading hypothesis~(\ref{hyp-reg}).
How can such a \emph{linear} regularity hypothesis apply to the
\emph{nonlinear} collision kernel $Q(\cdot,\cdot)$?
The trick is that, because of positivity and conservation of mass
(see~\S\,\ref{parCLCTE}), one can consider the Markov process as restricted
to the space of probability measures (with $r$-th polynomial moments),
which is some closed subset of an affine $\Hms$ space
(recall that a probability measure itself is not an element of $\Hms$).

\begin{Lem}\label{lem-faceadelta}
If, for one arbitrary (then for all) $v \in \RR^d$,
the linear operator $Q(\delta_v,\cdot) : \Hms\to\Hms$
satisfies the ``contractivity'' property
\[\label{for488} \forall f \in \Hms \quad \langle Q(\delta_v,f), f \rangle \leq -\frac{\kappa}{2} \|f\|^2 ,\]
then the restriction of $Q$ to the probability measures with $r$-th moment
satisfies hypothesis~(\ref{hyp-reg}).
\end{Lem}
\begin{proof}
Hypothesis~(\ref{hyp-reg}) for the probability measures with $r$-th moment
just means that for all probability measures $\mu$,
for all $f \in \Hms$,
\[\label{for684} \langle 2Q(\mu,f),f \rangle \leq -\kappa\|f\|^2 .\]
That follows directly from~(\ref{for488}) by writing the integration formula
$\mu = \int_{\RR^d} \delta_v \dd\mu(v)$.
\end{proof}
\begin{Rmq}
It is not difficult to see
that conversely the best $\kappa$ possible in~(\ref{hyp-reg})
is \emph{exactly} the best $\kappa$ possible in~(\ref{for488}).
We do not prove it as it is not essential,
but it will be implicitly used in Remark~\ref{nnMaxbuggue}.
\end{Rmq}

\begin{Lem}\label{lem3971}
Recall definition~(\ref{fordefus}) of $\phi_s$.
Note $(\ast\phi_s)$ the convolution operator
\[{\setlength{\arraycolsep}{2pt}\begin{array}{rcl}
\ast\phi_s : \ \Hms &\to& L^2 \\ f &\mapsto& f\ast \phi_s \end{array}} .\]
Then $Q(\delta_v,\cdot) : \Hms\to\Hms$ satisfies property~(\ref{for488})
if and only if
\[(\ast \phi_s) \circ Q(\delta_v,\cdot) \circ (\ast \phi_s)^{-1} : L^2 \to L^2 \]
satisfies the same property in the space $L^2(\RR^d)$.
\end{Lem}
\begin{proof}
It follows directly from the isomorphism formula~(\ref{forpfffatigue}).
\end{proof}
\begin{Rmq}\label{nnMaxbuggue}
Now we can understand why $\kappa$ cannot be finite for a non-Maxwellian model:
suppose the model satisfies~(\ref{parametre-gamma}) with $g \neq 0$.
For $\lambda \in (0,+\infty)$ denote by $I_{\lambda}$
the homothety transforming $v$ into $\lambda v$. Then you get
\[ Q \big( \delta_{0}, I_{\lambda}\mesim \mu \big)
= \lambda^{g} I_{\lambda}\mesim Q \big( \delta_{0}, \mu \big) ,\]
so if $Q(\delta_{0},\cdot)$ were $\kappa$-contracting for a $\kappa<0$
it would also be $\lambda^{g}\kappa$-contracting for all $\lambda$,
thus $0$-contracting, which is impossible.
\end{Rmq}

\subsection{Effective computation}

\begin{Lem}\label{lemcerceau}
Let $\theta \in [0,\pi]$; define a linear operator $\check{Q}_{\theta}$
on the set of measures on $\RR^d$, such that $\check{Q}_{\theta}(\delta_v)$
is the uniform probability measure on the $(d-2)$-dimensional sphere%
\footnote{That sphere degenerates into a point if $\theta \in \{0,\pi\}$.}
of velocities $v'$ such that $|v'-v/2| = |v|/2$ and $\hat{v\frac{v}{2}v'} = \theta$.
Then
\[ (\ast\phi_s) \circ \check{Q}_{\theta} = \big( \cos(\theta/2) \big)^s \check{Q}_{\theta} \circ (\ast\phi_s) .\]
\end{Lem}
\begin{Rmq}
$\check{Q}_{\theta}(\delta_v)$ respresents the post-collisional distribution of velocity
of a particle at initial velocity $v$ which has collided with a particle at initial velocity $0$,
undergoing an angular deviation $\theta$ in the collision referential,
the precise direction of that deviation being random.
\end{Rmq}
\begin{proof}
First let us give a neat proof working when $d$ is even.
Call $\mc{R}_{\theta}$ the set of the rotations $R$ of $\RR^d$ satisfying
$v\hspace{-2em}\hat{\hspace{2em}0\hspace{2em}}\hspace{-2em}(\!Rv\!) = \theta/2$ for \emph{all} $v \in \RR^d$.
If $d$ is even, $\mc{R}_{\theta}$ is non-empty
and has some canonical probability measure $\pi_{\theta}$ equipping it.
Then we notice that
\[\label{Qthetacommeint} \check{Q}_{\theta}(\mu) = \int_{\mc{R}_{\theta}} [\cos(\theta/2)R]\mesim \mu \ \dd\pi_{\theta}(R) .\]
Because of the rotational invariance of $\phi_s$, for any $R \in R_{\theta}$,
\[\label{commusA} (\ast\phi_s) \circ R\mesim = R\mesim \circ (\ast\phi_s) .\]
Similarly, the scale invariance of $\phi_s$ makes that for any $\lambda\in(0,\infty)$,
\[\label{commusB} (\ast\phi_s) \circ I_{\lambda}\mesim = \lambda^s \cdot I_{\lambda}\mesim \circ (\ast\phi_s) .\]
The result then follows by applying Formulae~(\ref{commusA}) and~(\ref{commusB})
to the integral~(\ref{Qthetacommeint}).

When $d$ is odd unfortunately I have nothing better than a calculation%
---which also works for even $d$.
Choose an arbitrary $v>0$, we will prove that $(\check{Q}_{\theta}\delta_v) \ast \phi_s =
\check{Q}_{\theta} (\delta_v\ast\phi_s)$, where $v$ denotes the point
$(v,0,\ldots,0)\in\RR^d$ whenever that makes sense.
Since these two functions
are invariant by any rotation around $v$, we will locate a point in $\RR^d$
merely by its first coordinate $z$ and its distance $\rho$ to the $z$ axis;
we will also denote $Z = \sqrt{z^2+\rho^2}$ its distance to $0$.
In the following calculations $\SS$ denotes the unit sphere in $\RR^{d-1}$,
equipped with its Lebesgue probability measure $\sigma$,
and $\rho$ denotes the point $(\rho,0,\ldots,0) \in \RR^{d-1}$ whenever that makes sense;
points of $\SS$ are denoted $y=(y_0,y_1)$ with $y_0 \in \RR$,
$y_1 \in \RR^{d-2}$.
Treating $(\check{Q}_{\theta}\delta_v)\ast\phi_s$ as a function, we find:
\begin{multline}\label{for434}
\big((\check{Q}_{\theta}\delta_v)\ast\phi_s\big)(z,\rho) \\ =
\int_{\SS} \Big\{ \big(\cos(\theta/2)v-z\big)^2 + \big(\sin\theta\,y_0v/2-\rho\big)^2 + (\sin\theta)^2|y_1|^2v^2/4 \Big\}%
^{-(d-s)/2}
\dd\sigma(y_0,y_1) \\
= \int_{\SS} \Big\{ Z^2 + \cos(\theta/2)^2 v^2
- 2\cos(\theta/2)^2vz - \sin\theta\,v\rho y_0 \Big\}^{-(d-s)/2}
\dd\sigma(y_0,y_1) .
\end{multline}
For $\check{Q}_{\theta} (\delta_v\ast\phi_s)$ it is more complicated
since that case needs computing a expression of type $\check{Q}_{\theta}f$,
$f$ being a function.
Usually that kind of computation raises no difficulty,
but here the operator $\check{Q}_{\theta}$ has some singularity
which makes it less tractable:
in $\check{Q}_{\theta}f$, the ``mass'' (in the measure sense) received by the point $(z,0,\ldots,0)$
comes only from a $(d-2)$-dimensional sphere in $\RR^d$%
---more precisely the sphere of points $(z,\rho)$,
$\rho \in \RR^{d-1}$, with $|\rho| = \tan(\theta/2) z$.
That regularity problem can be
overcome by an approximation technique, yielding:
\[ \big( \check{Q}_{\theta}f \big)(z,0,\ldots,0)
= \frac{1}{\cos(\theta/2)^d} \int_{\SS} f\big(z,[\tan(\theta/2)z] y\big) \dd\sigma(y) \]
---that formula also allowing to compute $\check{Q}_{\theta}f$
at points not located on the $z$ axis
by rotational invariance.

So
{\small
\begin{multline} \big(\check{Q}_{\theta} (\delta_v\ast\phi_s)\big)(z,\rho)
= \cos(\theta/2)^{-d} \cdot \\ \int_{\SS}
\Big\{ \big( z-\tan(\theta/2)\rho y_0-v \big)^2
+ \big( \rho+\tan(\theta/2) z y_0 \big)^2
+ \tan(\theta/2)^2 Z^2 |y_1|^2 \Big\}^{-(d-s)/2} \dd\sigma(y_0,y_1) \\
= \cos(\theta/2)^{-d} \int_{\SS}
\Big\{ \big(1+\tan(\theta/2)^2\big) Z^2 - 2v\big(z-\tan(\theta/2)\rho y_0\big) + v^2 \Big\}%
^{-(d-s)/2} \dd\sigma(y_0,y_1) \\
= \cos(\theta/2)^{-s} \big((\check{Q}_{\theta}\delta_v)\ast\phi_s\big)(z,\rho).
\end{multline}}
\end{proof}

\begin{Cor}\label{cor3982}
Let $Q_{\theta} = \check{Q}_{\theta} + \check{Q}_{\pi-\theta} - \check{Q}_{0} - \check{Q}_{\pi}$.
Then, for any $f \in \Hms$,
\[ \langle Q_{\theta}f,f \rangle \leq \big[ \big(\cos(\theta/2)\big)^r + \big(\sin(\theta/2)\big)^r - 1 \big] \cdot \|f\|^2 .\]
\end{Cor}
\begin{proof}
Observe first that $\check{Q}_{0}$ is the identity and that $\check{Q}_{\pi}=0$,
so it suffices to prove that the operator norm of
$\check{Q}_{\theta}$ in $\Hms$ is bounded above by $\big(\cos(\theta/2)\big)^r$.
Because of isomorphism formula~(\ref{forpfffatigue}),
that is also the norm of $(\ast\phi_s)\circ\check{Q}_{\theta}\circ(\ast\phi_s)^{-1}$
in $L^2$, which is $\cos(\theta/2)^s \check{Q}_{\theta}$ by the Lemma~\ref{lemcerceau}.
Thus we just have to bound the norm of $\check{Q}_{\theta}$,
\emph{regarded as an operator in $L^2$}, by $\cos(\theta/2)^{-d/2}$.
Now we note that one can write
\[\check{Q}_{\theta}f =  I_{\cos(\theta/2)}\mesim \tilde{Q}_{\theta}f ,\]
where $\tilde{Q}_{\theta}$ is the kernel of the Markov chain on $\RR^d$
which sends $x$ uniformly to the $(d-2)$-dimensional sphere
of points $y$ such that $|y| = |x|$ and $\hat{x0y} = \theta/2$.
But that Markov chain has the Lebesgue measure on $\RR^d$
as reversible equilibrium measure, thus
$|\!|\!|\tilde{Q}_{\theta}|\!|\!|_{L^2} \leq 1$,
so that $|\!|\!|\check{Q}_{\theta}|\!|\!|_{L^2} \leq \cos(\theta/2)^{-d/2}$,
\emph{quod erat demonstrandum}.
\end{proof}

Now we are ready to state the main result of this section:
\begin{Thm}\label{thm-valeur-k}
In a Maxwellian model, calling $\bar{\gamma}$ the common value of
all the measures $\bar{\gamma}_{u}$, the collision kernel $Q$,
when restricted to the probability measures, satisfies hypothesis%
~(\ref{hyp-reg}) with
\[ \kappa = \int_0^{\pi} \big[ 1 - \cos(\theta/2)^r - \sin(\theta/2)^r \big] \, \dd\bar{\gamma}(\theta) .\]
\end{Thm}
\begin{proof}
Note that
\[Q(\delta_0,\cdot) = \frac{1}{2} \int_0^{\pi} Q_{\theta} \dd\bar{\gamma}(\theta) \]
and apply all the previous work of this section (Lemmas~\ref{lem-faceadelta},
\ref{lem3971}, \ref{lemcerceau} and~\ref{cor3982}).
\end{proof}

\begin{Xpl}
The ``Kac'' model\footnote{Actually this is not exactly the Kac model~\cite{Kac53},
but the spirit is the same.}
is the case when the measure $\gamma_{v,w}$
is always uniform on the sphere supporting it with total mass $1$,
i.e.\ it is a Maxwellian model with
\[\label{for-mod-Kac}
\dd\bar{\gamma}(\theta) = \frac{\Gamma(d-1)}{2^{d-2}\Gamma\big((d-1)/2\big)^2} (\sin\theta)^{d-2} \dd\theta .\]
For that model one has $-\infty<\kappa<0$ for any $r\in(0,1)$.
\end{Xpl}
\begin{Xpl}
The model of Maxwellian potential corresponds to particles having a repulsive force
with a radially symmetric potential decreasing as $\rho^{-(2d-2)}$
as the radius $\rho$ between the two particles increases.
For that model the measure $\bar{\gamma}$ is not finite
because then one has $\dd\bar{\gamma} \sim \theta^{-3/2}\dd\theta$
when $\theta\to 0$ for any $d$.
Yet it remains possible to define both the $N$-particle and the limit models,
see Remark~\ref{RmqsmodeleMarkov}-\ref{itmintblt}.
For that model, one also has $-\infty<\kappa<0$ for any $r\in(0,1)$.
\end{Xpl}

\section{Initial value}\label{flucCI}

In Formula~(\ref{forDUthm}) given by Theorem~\ref{ZEthm},
as we saw, besides the dynamic term there is a term 
due to the fluctuations of the initial empirical measure.
In this section we control these fluctuations
in the case of i.i.d.\ initial particles.

Let $\mu$ be a probability measure on $\RR^d$ and let $r\in(0,1)$.
We assume that $\mu$ has an $r$-th exponential moment,
i.e.\ that there exists some $a>0$ such that
\[ \label{exprintblt} \int_{\RR^d} e^{a|v|^r} \dd{v}<\infty.\]
In the sequel we assume $a$ to be fixed.

If $v$ is a random variable of $\RR^d$ with law $\mu$,
then $\delta_v - \mu$ is a random variable in $\Hms$,
whose law will be denoted by $\mathfrak{D}_{\mu}$:
$\mathfrak{D}_{\mu}$ is a centered probability measure on $\Hms$.
I claim that $\mathfrak{D}_{\mu}$
has an exponential moment with parameter $a$, i.e.
\[ \label{expaintblt} \int_{\Hms} \! e^{a\|\nu\|} \, \dd\mathfrak{D}_{\mu}(\nu) < \infty :\]
To prove it it suffices to note that
\[ \| \delta_v - \mu \| \leq \| \delta_v - \delta_0 \| + \| \delta_{v_0} - \mu \|
= C_r |v-v_0|^r + \| \delta_{v_0} - \mu \|, \]
whose $a$-parameter exponential is integrable because of~(\ref{exprintblt}).

So the law $\mathfrak{D}_{\mu}$ has a finite exponential moment,
\emph{a fortiori} a finite variance.
Let us denote it by $\sigma^2$:
\[ \sigma^2 = \int_{\Hms} \|\nu\|^2 \dd\mathfrak{D}_{\mu}(\nu) .\]

Now we have all the definitions at hand to state the main result of this section:
\begin{Thm}\label{thm-empirique}
Let $v_0,\dots,v_{N-1}$ be $N$ i.i.d.\ random variables on $\RR^d$ with law $\mu$,
and denote $\hat{\mu}^N = N^{-1} \sum_{i=0}^{N-1} \delta_{v_i}$
their empirical measure.
Then there exists an explicit constant $A(\mu)$,
which is easy to bound,
such that for all $\lambda \leq aN$:
\[\label{for-empirique} \EE\big[ \mc{U} \big( \lambda (\hat{\mu}^N - \hat{\mu}) \big) \big]
\leq 2 \exp\bigg( \frac{\lambda^2\sigma^2}{2N} + \frac{\lambda^3 A(\mu)}{N^2a^3} \bigg) .\]
\end{Thm}
Before proving Theorem~\ref{thm-empirique},
let us further examine Formula~(\ref{for-empirique}):
the term in the exponential remains bounded when $N\to\infty$
if $\lambda$ increases as $N^{1/2}$, like in~(\ref{approx1}).
Thus, writing $\lambda=yN^{1/2}$ like in~(\ref{for783}):
\[ \EE\big[ \mc{U} \big( yN^{1/2} (\hat{\mu}^N - \mu) \big) \big]
\leq 2 \exp \bigg( \frac{\sigma^2y^2}{2} + \frac{A(\mu)y^3}{a^3}N^{-1/2} \bigg) .\]
Though we will not use it in the sequel, note the following
\begin{Cor}\label{cor-empirique}
For $S\geq\sigma^2$, for all $x \geq 0$, for all $N\geq N_0 \egdef x^2/a^2S^2$:
\[ \label{DV-empirique} \PP\Big( \|\hat{\mu}^N - \mu\| \geq xN^{-1/2} \Big)
\leq \exp \bigg( -\frac{x^2}{2S} + \ln 2 + A(\mu){N_0}^{1/2}N^{-1/2} \bigg) .\]
\end{Cor}
\begin{Rmq}
(\ref{DV-empirique}) works as soon as $N\geq x^2/a^2S^2$,
i.e.\ as soon as $x \leq aSN^{1/2}$,
so that estimate is valid up to the large deviations setting.
\end{Rmq}

{\renewcommand{\proofname}{Proof of Theorem~\ref{thm-empirique}}
\begin{proof}
The principle of the proof is exactly the same as for Theorem~\ref{ZEthm},
except that here time will be discrete.

Let $v_0,\ldots,v_{N-1}$ be $N$ i.i.d.\ random variables with law $\mu$
and set $\hat{M}_i = \sum_{j=0}^{i-1} N^{-1}(\delta_{v_i} - \mu)$,
then $(\hat{M}_i)_i$ is a Markov chain and a martingale, and $\hat{M}_N$
has the same law as $\hat{\mu}^N$.
We are going to prove that
\[\label{recurrenceMarkov} \EE\big[ \mc{U}(\lambda \hat{M}_{i+1}) \big| \hat{M}_{i} \big]
\leq \exp\bigg( \frac{\lambda^2\sigma^2}{2N} + \frac{\lambda^3 A(\mu)}{a^3N^2} \bigg)
\cdot \mc{U}(\lambda \hat{M}_i) ,\]
whence the result.

To get~(\ref{recurrenceMarkov}), thanks to Lemma~\ref{lemme3.3}
it suffices to prove that
\[\label{for131} \int \big( e^{\lambda N^{-1}\|\nu\|} - N^{-1}\lambda N^{-1}\|\nu\| \big) \,\dd\mathfrak{D}_{\mu}(\nu)
\leq \exp\bigg( \frac{\lambda^2\sigma^2}{2N^2} + \frac{\lambda^3 A(\mu)}{a^3N^3} \bigg) .\]
We set
\[\label{for-def-A} A(\mu) = \int \big(e^{a\|\nu\|} - a\|\nu\| - 1\big) \dd\mathfrak{D}_{\mu}(\nu) .\]
The function $e_2(t) = (e^t-1-t)/t^2$ is convex on $\RR_+$,
so
\[ \forall t \geq 0 \quad \forall \theta \in [0,1] \qquad
e^{\theta t} - \theta t - 1 \leq \frac{1}{2}(1-\theta)\theta^2 t^2 + \theta^3 (e^t-t-1) .\]
Consequently
\[ \int \big( e^{N^{-1}\lambda\|\nu\|} - N^{-1}\lambda\|\nu\| - 1 \big) \,\dd\mathfrak{D}_{\mu}(\nu) \leq
\big(1-\lambda aN^{-1}\big) \frac{\lambda^2\sigma^2}{2N^2} + \frac{\lambda^3 A(\mu)}{a^3 N^3} ,\]
whence~(\ref{for131}).
\end{proof}}

\section{Discussion}

\subsection{Examples of synthetic results}

Until now in our article we have just given separate results,
mainly Theorem~{\ref{ZEthm}}, Formulae~(\ref{bonnevaleurL})
and~(\ref{bonnevaleurV}), and Theorems~\ref{thm-valeur-k}
and~\ref{thm-empirique}. It is obvious that all these results
are to be put together to get synthetic results
on the convergence of $N$-particle dynamic models
to their mean field limit; yet we have not done it in the previous sections.

There are several reasons why I have postponed the presentation
of such synthetic results to the last section.
The most obvious one is that these global results would have been quite unreadable
if put in the beginning of the article.
More important, the different ``bricks'' of results
given within the core of the paper are open to improvements different for each,
some of which may work for some cases but not for others,
so that there may be no optimal general result.

Let us however give some examples of formulae
got by piling our theorems together%
---proofs will not be given since they really consist in some plain gluing game:
\begin{Thm}\label{thm-synthese}
Let $d\geq2$, $r\in (0,1)$.
Let $\mu_0$ be a probability measure on $\RR^d$
with finite $r$-exponential moments for all $r<1$.
Up to translating the origin of $\RR^d$ we can suppose that
$p \egdef \int_{\RR^d} v\dd\mu_0(v) = 0$;
then let $k = \frac{1}{2} \int |v|^2 \dd\mu_0(v)$.
Choose some $k_1 > k$ and define
\begin{eqnarray}
\kappa &=& 1 - \frac{\Gamma(d-1)}{2^{d-3} \Gamma\big((d-1)/2\big)^2} \int_{0}^{\pi} \sin(\theta/2)^r \sin(\theta)^{d-2}
\dd\theta \,\footnotemark, \\
\ell &=& 2^{1+r} \sqrt{2^{1-r}-1} C_r k_1^{r/2} , \\
\omega &=& (2^{1-r}-1)2^{1+2r} {C_r}^2 k_1^r , \\
\sigma^2 &=& \int_{\RR^d} \| \delta_v - \mu_0 \|_{\Hms}^2 \dd\mu_0(v) .
\end{eqnarray}
\footnotetext{Beware that $\kappa<0$.}

Let $N\geq2$; let $v_0, \ldots, v_{N-1}$ be $N$ i.i.d.\ random variables
with law $\mu_0$ and let $\hat{\mu}^N_0$ be their empirical measure;
denote $\hat{K}^N = \frac{1}{2} \sum_{i=0}^{N-1} |v_i|^2$.
Let $\hat{\mu}^N_t$ be the empirical measure at time $t$
of the Markov process with generator~(\ref{generateurMarkov})
for the ``Kac'' model~(\ref{for-mod-Kac})
and initial condition $(v_0,\ldots,v_{N-1})$.
Let $(\mu_t)_{t\geq0}$ be the deterministic evolution%
~(\ref{eq-Boltz}) for the same model with initial value $\mu_0$.

Then for any $a>0$, there is a (easily bounded) constant $A(a,\mu)$
such that, for any $T>0$,
as soon as $\lambda \leq a e^{-|\kappa|T} N$:
\begin{multline}\label{megaformule}
\ln \EE \big[ \1_{\hat{K}^N\leq Nk_1} \mc{U}\big( \lambda(\hat{\mu}^N_t - \mu_t) \big) \big] \leq \\
\ln 2 + \frac{e^{2|\kappa|T}\lambda^2\sigma^2}{2N} + \frac{e^{2|\kappa|T}\lambda^3A(a,\mu)}{N^2a^3}
+ \frac{\lambda^2 \omega T}{N} e_1(2|\kappa|T) e_2\big( \lambda e^{2|\kappa|T} \ell N^{r/2-1} \big) .
\end{multline}
\end{Thm}
\begin{Cor}
For the same model, for any $y \geq 0$:
\[\label{for578aaa}
\limsup_{N\to\infty} \ln \EE \big[ \1_{\hat{K}^N\leq Nk_1} \mc{U}\big( yN^{1/2}(\hat{\mu}^N_t - \mu_t) \big) \big]
\leq \ln 2 + e^{2|\kappa|T}\frac{\sigma^2y^2}{2} + e_1(2|\kappa|T)\frac{\omega T y^2}{2} .\]
\end{Cor}
\begin{Cor}
Still for the same model, for any $x \geq 0$:
\[\label{for578} \limsup_{N\to\infty} \PP\big( \| \hat{\mu}^N_t - \mu_t \| \geq x N^{-1/2} \big)
\leq 2 \exp\bigg( \frac{-x^2}{2 \big[ e^{2|\kappa|T} \sigma^2 + e_1(2|\kappa|T) \omega T \big]} \bigg) .\]
\end{Cor}
\begin{Rmq}
As~(\ref{for578}) is true for any value of $k_1$, we can make $k_1$ approach $k$ in it,
which allows to replace $\omega$ by $\omega_0 \egdef (2^{1-r}-1)2^{1+2r} {C_r}^2 k^r$.
\end{Rmq}

\subsection{Optimality}

Theorem~\ref{thm-synthese}
essentially gives us a convergence to the continuous limit
at rate $N^{-1/2}$ with gaussian control.
Qualitatively it is the best result one can hope for,
because it is the same way of convergence
as for central limit theorems.
Quantitatively however, is the parameter in the Gaussian bound optimal?

Here we will look at what happens for Theorem~\ref{thm-empirique}
(Theorem~\ref{ZEthm} exhibits the same behaviour,
but it is harder to see).
Through Corollary~\ref{cor-empirique}, Theorem~\ref{thm-empirique} gives
some Gaussian bound in an infinite-dimensional frame.
Yet its proof, whose main ingredient is the use of the utility function $\mc{U}$,
would work as well in a finite-dimensional setting.
So let us imagine that we replace $\Hms$ by $\RR^d$
and $\mathfrak{D}_{\mu}$ by the centered normal law with variance $\mathbf{I}_d$,
denoted by $\mathcal{N}$;
then $\sigma^2$ becomes $\int_{\RR^d} |x|^2 \dd\mathcal{N}(x) =d$.
In that case $\hat{\mu}^N$ turns into
a random variable $X^N$ on $\RR^d$ which is centered normal with variance $N^{-1}I_d$,
and we get:
\[ \PP(X^N \geq xN^{-1/2}) \leq \exp\left( -\frac{x^2}{2d} + \ln 2 + A{N_0}^{1/2}N^{-1/2} \right) \]
for some $A$ and $N_0$ not depending on $N$, so making $N\to\infty$:
\[ \mc{N}(|X|\geq x) \leq 2 e^{-x^2/2d} ,\]
whereas the exact result is
\[ \mc{N}(|X|\geq x) = \frac{2^{1-d/2}}{\Gamma(d/2)} \int_x^{\infty} y^{d-1} e^{-y^2/2} \dd{y}
\underset{x\to\infty}{\approx} e^{-x^2/2} ,\]
where ``$\approx$'' means ``having equivalent logarithms''.

So for $d>1$ the parameter in the gaussian bound is underestimeted
by a factor $d$. Why that? Well, the proof of Theorem~\ref{thm-empirique}
uses the bound on the curvature of the utility function $\mc{U}$
given by Claim~\ref{proprietes-de-U}-\ref{D2U}.
But as soon as $x$ is reasonably large,
the Hessian of $\mc{U}$ at $x$ is much more curved in one direction
than in all the other ones, so that Formula~(\ref{forlemme3.3}) in Lemma~\ref{lemme3.3}
becomes strongly suboptimal since the factor $\|y\|^2$ in it should morally be replaced
by the sole component of the variance of $y$ along the direction along which
$\nabla^2\mc{U}(x)$ is most curved.

So the techniques involving $\mc{U}$ are poor as soon as the dimension
in which the random phenomena occur becomes large.
For our particle models we work in $\Hms$, whose dimension is\dots{} infinite!
Does that mean that our results are ``infinitely bad''?
Actually not, because each increment of the martingale $\hat{M}$
(see the proof of Theorem~\ref{thm-empirique}) is determined by the value of one $v_i$,
so the law of this increment may be seen as a probability on $\RR^d$.
More precisely, the support of $\mathfrak{D}_{\mu}$
is isometric to $\RR^d$ equipped with the distance $|\cdot-\cdot|^r$,
whose Hausdorff dimension is $d/r$, so that
the ``effective'' dimension of $\Hms$ in our theorem is $d/r$.

As a consequence we had better not choose $r$ too close to $0$.
On the other hand, the bigger $s$ is, the more regular
the test functions in the definition of $\|\cdot\|_{\Hms}$
(see Proposition~\ref{Hms-comme-dual}) will be,
so the less small-scale details $\|\cdot\|_{\Hms}$ will catch%
\footnote{Remember however that homogeneous Sobolev spaces
have no inclusion relations.}.
So it should be advised to take medium values of $r$,
e.g.\ $r=1/2$.

\subsection{A numerical computation}

One important side of our work is that it gives non-asymptotic results.
The idea behind it is that, to understand Boltzmann's evolution,
we will not actually look at $N\to\infty$, but rather take some \emph{fixed} large $N$
and say that the behaviour of the $N$-particle system for that $N$
is very close to the limit evolution with very large probability.
In particular, think about the case of numerical simulation: we cannot afford
dealing with $10^{24}$ particles on our computers!

Here I will compute numerical values for the following case:
the collision kernel is the one of ``Kac'' model for $d=3$
and we take $\mu_0 = \frac{1}{2} (\delta_{-1}+\delta_{1})$.
Physically speaking, it means that we crash together
two same-sized sets of frozen particles with relative speed $2$.
Then the collisions between particles of differents sets
will tend to scatter the distribution of velocities of the particles,
which will morally converge to the law~(\ref{mueq}) with $k=1/2$
in a few units of time%
---this is the behaviour of Boltzmann's equation~(\ref{eq-Boltz}) indeed.
The question is, which $N$ will we choose to be almost certain
that the evolution of the particle system will be fairly close to%
~(\ref{eq-Boltz})?

Say we take $r=1/2$ and we want to have $\| \hat{\mu}^N_T - \mu_T \|_{\Hms}$
greater than $\epsilon = 10^{-2}$ with probability less than $q = 10^{-1}$ for $T=3$.
As in our case $\hat{K}^N \leq Nk$ almost surely, we take $k_1 = k = 1/2$.
Then one computes the following numerical values, which are all rounded above:
\begin{eqnarray}
-\kappa & \simeq & 0.600; \\
\ell &\simeq& 0.432; \\
\omega &\simeq& 0.0933; \\
\sigma^2 &\simeq& 0.0398.
\end{eqnarray}
We choose arbitrarily $a=1$, then~(\ref{for-def-A}) gives
$A(a,\mu) \simeq 0.0213$.
We have to take $\lambda \gtrsim |\ln q|/\epsilon$,
so let us put $\lambda = 500$.
For $N=8\cdot10^5$ we find by~(\ref{megaformule}):
\[ \ln \EE \big[ \mc{U}\big( \lambda(\hat{\mu}^N_t - \mu_t) \big) \big] < 2.692 ,\]
thus
\[ \PP \big[ \|\hat{\mu}^N_t - \mu_t\| \geq 10^{-2} \big] < 10^{-1} \]
by Markoff's inequality.

So with a discrete system of $8\cdot10^5$ particles
one will much probably find a quite good approximation of
Boltzmann mean field limit by running the particle system
$3$ units of time.
Now simulating $8\cdot10^5$ particles is easy for today's computers,
which shows that our bounds can actually be useful in practice.
However there is little doubt that the true speed of the mean field convergence
is much faster than what our computations suggest.

\begin{Rmq}
Here we have bypassed the problem of the $\1_{\hat{K}^N\leq Nk_1}$ factor
in~(\ref{megaformule}) by a specific argument. How can we do for it in the general case?
Well, merely note that, as soon as one wants to have a result in terms of probability,
they will just have to add $\PP(\hat{K}^N > Nk_1)$
to the probability they get forgetting the indicator.
But the event $\{\hat{K}^N > Nk_1\}$ is a large deviations event,
so as soon as $\mu_0$ has some square-exponential moment
its probability will decrease exponentially with $N$
and thus cause no problem actually.
\end{Rmq}

\subsection{Comparison to older results}\label{parcomparaison}

The ``usual'' method to tackle mean field limit problems relies
on the concept of \emph{propagation of chaos},
devised by Kac~\cite{Kac53}:
briefly speaking, one says that there is some ``chaos'' in a large assembly of particles
if there is no correlation between the states of the particles,
i.e.\ if the states of the particles are approximately i.i.d.%
\footnote{Propagation of chaos may be seen as an attempt to make
Boltzmann's \emph{Sto\ss{}zahlansatz} (cf.~\cite{Boltzmann}\nocite{Boltzmann-traduction}) rigorous.}.
Then, if one manages to prove some propagation of chaos,
one can study only the probability distribution of \emph{one} particle
to understand the behaviour of the whole assembly.
That idea is linked with the concept of \emph{nonlinear particle},
due to Sznitman~\nocite{Sznitman-article}\cite{Sznitman-SF},
which is a non-homogeneous Markov process
describing the asymptotic behaviour of \emph{one} particle
when the number of particles becomes large,
leading to a \emph{trajectorial} version of propagation of chaos.

Thanks to these tools Sznitman~\cite{Sznitman-article}\nocite{Sznitman-SF}
proved propagation of chaos
for spatially homogeneous Boltzmann models,
thus giving a qualitative, asymptotic version of our Theorem~\ref{thm-synthese}.
More recently Graham and M\'el\'eard~\cite{GM}
proved propagation of chaos for the more general \emph{Povzner equation},
which is a kind of mollified spatially heterogeneous Boltzmann equation
where each particle can collide with any particle
in a small spatial neighbourhood around it.
Yet these essentially qualitative methods hardly give quantitative results.

My paper was motivated by the reading of~\cite{BGV},
in which Bolley \emph{et al.}
tackle some mean field limit problems in a quantitative way
by working with $W_1$ Wasserstein distances for the empirical measures.
They get an explicit control on the large deviations of the difference
between the empirical measure of the $N$-particle system and its theoretical limit
for positive times. Yet there are two annoying shortcomings in their work:
\begin{itemize}
\item First, it seems to be limited to McKean--Vlasov models,
that is, systems where the interactions between particles are due to forces rather than collisions.
In fact the proof fundamentally relies on a coupling technique
(popularized by Sznitman~\cite{Sznitman-article}\nocite{Sznitman-SF}, Dobrushin~\cite{Dobrushin} and others),
in which one defines a coupling between the real assembly of particles
and a virtual assembly of $N$ independent nonlinear particles.
Such a technique has little relevance when one deals with collisions,
because these events imply \emph{two} particles at the \emph{same} moment each time they occur,
so there is no natural way of coupling.
\item Secondly, the results of~\cite{BGV} are good for large deviations,
but the control they give for medium deviations is far too poor to get,
as we would wish, some $N^{-1/2}$ convergence rate.
As we told in~\S\,\ref{Why}, that is actually an intrinsic shortcoming of
$W_1$ distances.
\end{itemize}

The idea of studying the empirical distribution of the particles in some Hilbert space
to bypass the coupling problems
is due to Fernandez and M\'el\'eard~\cite{FM,Meleard},
who wanted to study the problem of fluctuations of the particle evolution,
a problem linked to the uniform central limit theory~\cite{Dudley}.
Yet I did not know their work by the time I wrote this paper,
and actually the framework here is quite different from that of~\cite{FM}.
To the best of my knowledge, my article is the first one
to use Hilbert spaces in a collision model,
and moreover in a non-asymptotic setting.

To finish with this overview
I must mention the current, independent work~\cite{MMW} of Mischler, Mouhot and Wennberg,
in which they devise a technique studying directly
the evolution of the empirical measure of the system in some Banach space.
Their very general framework is liable to apply to a wide range of situations,
though its quantitative version does not seem to give
the optimal $N^{-1/2}$ rate of convergence.

\subsection{Uniform in time bounds}

The results we have given work for some \emph{fixed} $T$,
i.e.\ they control $\| \hat{\mu}^N_T - \mu_T \|$.
Yet it may be more natural to control
$\sup_{t\in[0,T]} \| \hat{\mu}^N_t - \mu_t \|$,
i.e.\ to say that the system is \emph{always}
close to the Boltzmann mean field limit
between times $0$ and $T$,
as~\cite{BGV} does for McKean--Vlasov models.
We will not do it here, but note that,
as we have used martingale techniques, getting results
valid for all $t\in[0,T]$ can be easily achieved
from the previous work by Doob's inequality~\cite[\S\,II-1.7]{RY}.
Actually for $\kappa\leq0$ it would turn out that uniform in time results
are not much different from fixed time results,
which is quite logical
because then the control on $\|\hat{\mu}^N_t - \mu_t\|$
is worst for $t=T$.
For $\kappa>0$ yet, when $T$ is large
the maximum of the difference between $\hat{\mu}^N_t$ and $\mu_t$
is much less well controlled than its terminal value,
as we already noticed in Footnote~\ref{notek>0} at page~\pageref{notek>0}.

\subsection{Going further}

I find my results rather disappointing
in their present state, but I am still working on them,
hoping to get substantial improvements, e.g.\
having cases for which $\kappa>0$ or theorems applying to non-Maxwellian models.
An important tool seems to be the choice of good functional spaces
to work in rather than $\Hms$, but it is still current work.

\def\cprime{$'$}

\end{document}